\documentclass[11pt]{article}
\usepackage{amsfonts,amsmath,amsthm,stmaryrd}
\usepackage{amssymb,amsmath,amsthm, enumerate}
\usepackage{latexsym}
\usepackage{fullpage}
\usepackage{tikz}
\usetikzlibrary[patterns]
\tikzstyle{ann} = [fill=white,inner sep=1pt]

\usepackage{hyperref}
\usepackage{graphicx}
\usepackage{float}

\newtheorem{theorem}{Theorem}[section]

\newtheorem{question}{Question}[section]

\newtheorem{lemma}[theorem]{Lemma}
\newtheorem{remark}[theorem]{Remark}

\newtheorem{cor}[theorem]{Corollary}
\newtheorem{defn}[theorem]{Definition}

\def\Sum{\sum\nolimits}
\def\Prod{\prod\nolimits}

\def\P{\mathbb{P}}
\def\E{\mathbb E}

\def\1{{\bf 1}}

\newcount \questionno
\def\question{\medbreak
        \global \advance \questionno 1
        \textbf{Problem \the\questionno}.\enspace \ignorespaces}
\newcommand{\remove}[1]{}
  
  \DeclareMathOperator{\V}{Var} 
  \DeclareMathOperator{\Cov}{Cov} 

\title{Branching random walk in the presence of a hard wall}

\author{ Rishideep Roy\thanks{ Email: rishideeproy@gmail.com} }

\begin{document}
\maketitle  
\begin{abstract}
We consider a branching random walk on a $d$-ary tree of height $n$ ($n \in \mathbb{N}$), under the presence of a hard wall which restricts each value to be positive, where $d$ is a natural number satisfying $d\geqslant2$.  The question of behaviour of Gaussian processes with long range interactions, for example the discrete Gaussian free field, under the condition that it is positive on a large subset of vertices, and a relation with the expected maximum of the processes has been observed. We find the probability of the  event that the branching random walk is positive at every vertex in the $n^{th}$ generation, and show that the conditional expectation of the Gaussian variable at a typical vertex, under positivity, is less than the expected maximum by order of $\log n$.  
\end{abstract}

\section{Introduction} 

We consider a tree of $n$ levels, where $n \in \mathbb{N}$. We assume that the root node of this tree has $d$ number of children, who in turn have $d$ number of children each, and so on, till generation $n$, which are leaf nodes. Here we assume that $d \in \mathbb{N}$, with $d\geqslant 2$, since the case $d=1$ is really trivial. 
This is a $d$-ary tree, and we refer to it as $T^n$. We refer to the subset of all the leaf nodes of this tree by $T_n$. So $\mid T_n \mid = d^n$. 

We consider a particle starting from $0 \in \mathbb{R}$, which dies at time $1$, and splits into $d$ number of children. Each of these children travels a distance given by independent standard Gaussian random variables. Then at time $2$, each of these children dies, and give rise to $d$ number of children each, which in turn follow the same process, the displacements at each step being independent of the displacements in the previous time points.  At time $n$ we have $d^n$ many particles, each having a displacement. The positions of these particles are collectively called branching random walk at time $n$.
We can equivalently define a  branching random walk (BRW) as a Gaussian process on $T^n$. The origin stands for the root, and the displacements at each step are attached to the edges. So, the $d$ displacements at first generation are attached to the edges between the root and its $d$ children. To each vertex we attach a quantity equal to sum of the Gaussian variables that we encounter while looking at the shortest path between itself and the root. The collection of displacements of the $d^n$ particles at time $n$ is given by the Gaussian variables attached to all the vertices in $T_n$. We denote it by $\{ \phi_v^n : v \in T_n \}$. This is the branching random walk at time $n$. Two particles at time $n$ having the last common ancestor at time $k$ $(k \leq n)$ is equivalent to two leaf nodes, which have branched out from the same vertex at level $k$ of the tree. Each of these displacements are sums of $n$ independent standard Gaussian random variables. Figure \ref{BRWbinary} gives a pictorial representation of the branching random walk for the case $d=2$, i.e. on a binary tree.  The collection $\{ X_{i,j}: j=1,2,\ldots, d^i, i =1,2,\ldots, n \}$ represents independent displacements, and are i.i.d standard Gaussian random variables. There are $d^n$ many leaf nodes in the tree, and we can fix an ordering of the vertices from $1$ to $d^n$. For any $v \in T_n$, i.e. $v \in \{1,2,\ldots, d^n \}$, we define $a_i(v)=\lceil \frac{v}{d^{n-i}}\rceil$ for $i=1,2,\ldots, n$. Then we can define $\phi_n^v=\sum_{j=1}^n X_{j,a_j(v)}$. This is another way of constructing the BRW.

\begin{figure}[h!]\label{BRWbinary}
	\tikz[scale=0.75]{
		\draw[black, thick] (9,9) -- (6,6);
		\draw[black, thick] (6,6) -- (4.5,4);
		\draw[black, thick] (6,6) -- (7.5,4);
		\draw[black, dashed] (4.5,4) -- (3.5,2);			
		\draw[black, dashed] (4.5,4) -- (5.5,2);
		\draw[black, thick] (3.5,2) -- (4.5,0);
		\draw[black, thick] (3.5,2) -- (2.5,0);
		\draw[black, dashed] (7.5,4) -- (8.5,2);			
		\draw[black, dashed] (10,4) -- (9,2);
		\draw[black, dashed] (5.5,0) -- (12.5,0);
		\draw[black, dashed] (14,4) -- (16,2);			
		\draw[black, dashed] (14,4) -- (12.5,2);
		\draw[black, thick] (12,6) -- (14,4);
		\draw[black, thick] (12,6) -- (10,4);
		\draw[black, thick] (9,9) -- (12,6);
		\draw[black, thick] (16,2) -- (13.5,0);
		\draw[black, thick] (16,2) -- (17,0);
		\draw[black, thick, fill=white] (9,9) circle[radius=.45cm];				
		\draw[black, thick, fill=white] (6,6) circle[radius=.55cm];			
		\draw[black, thick, fill=white] (12,6) circle[radius=.55cm];
		\draw[black, thick, fill=white] (4.5,4) circle[radius=.45cm];
		\draw[black, thick, fill=white] (4.5,0) circle[radius=.45cm];				
		\draw[black, thick, fill=white] (2.5,0) circle[radius=.45cm];	
		\draw[black, thick, fill=white] (13.5,0) circle[radius=.45cm];				
		\draw[black, thick, fill=white] (17,0) circle[radius=.45cm];
		\draw[black, thick, fill=white] (14,4) circle[radius=.45cm];
		\node[ann] at (9,9 ) {$0$};	
		\node[ann] at (7.5,7.5) {$X_{1,1}$};
		\node[ann] at (6,6) {$X_{1,1}$};
		\node[ann,scale=.5] at  (4.5,4)  {$X_{1,1}+X_{2,1}$};
		\node[ann,scale=.5] at  (14,4)  {$X_{1,2}+X_{2,4}$};
		\node[ann] at (12,6 ) {$X_{1,2}$};
		\node[ann] at (5,5) {$X_{2,1}$};
		\node[ann] at (7,5) {$X_{2,2}$};
		\node[ann] at (11, 5) {$X_{2,3}$};	
		\node[ann] at (13, 5) {$X_{2,4}$};	
		\node[ann] at (3, 1) {$X_{n,1}$};	
		\node[ann] at (4, 1) {$X_{n,2}$};
		\node[ann] at (16.5, 1) {$X_{n,2^n}$};	
		\node[ann] at (14.5, 1) {$X_{n,2^n-1}$};
		\node[ann] at (2.5, -.65) {$1$};
		\node[ann] at (4.5, -.65) {$2$};
		\node[ann] at (13.5, -.65) {$2^n-1$};
		\node[ann] at (17, -.65) {$2^n$};
		\node[ann] at (2.5, 0) {$\phi^n_{1}$};	
		\node[ann] at (4.5, 0) {$\phi^n_{2}$};
		\node[ann] at (13.5, 0) {$\phi^n_{2^n-1}$};
		\node[ann] at (17, 0) {$\phi^n_{2^n}$};			
		\node[ann,scale=.75] at (0.75, 0) {$\phi^n_{1}=\sum_{j=1}^{n}X_{j,1}$};			
		\node[ann] at (10.25, 7.5) {$X_{1,2}$};		
	}
	\caption{BRW on  binary tree }
\end{figure}

The covariance structure of this Gaussian process is given by the following:
\begin{equation}\label{eq-brw-definition}
\begin{split}
\V (\phi_v^n) &= n \quad \text{for all } v \in T_n \\
\Cov(\phi_u^n, \phi^n_v) &=  n - \frac{1}{2}d_T(u,v) \quad \text{ for all }u \neq v \in T_n \,.
\end{split}
\end{equation}
where $d_T$ denotes the graph distance. So essentially the BRW is a multivariate normal distribution of dimension $d^n$, with all means $0$, and variances and covariances given by (\ref{eq-brw-definition}). We call the corresponding probability measure  $\mathbb{P}(\cdot)$. 

We wish to find bounds on the order of the probability of a branching random walk being positive at the leaf nodes ($v \in T_n$). This is also the event that all the particles are on the right of the starting point of the first particle. We also wish to find the expected value of the field at a typical vertex in generation $n$, under the condition that the BRW is positive at all the leaf nodes.
The behaviour that we are considering is that of entropic repulsion for this Gaussian field, which is its phenomenon of drifting away when pressed against a hard wall so as to have enough room for local fluctuations, as is referred to in \cite{lebowitz87}. The phenomenon of entropic repulsion for the Gaussian free field (GFF) has been studied in literature for some time now. The entropic repulsion for infinite GFF on $\mathbb{Z}^d$, $d\geq 3$ has been studied in \cite{bolthausen1995}. As a continuation to this, the GFF on a finite box with Dirichlet boundary conditions, for dimension $3$ or more was studied in \cite{deuschel96}. In case of the GFF on finite box, the positivity was looked at from two different angles, one involving the interior only, while the other considers the whole box. Both looked at the phenomenon of positivity of the field in a box of size $N$. Though the typical behaviour of a vertex was similar, this order was not so, when positivity for the entire box was considered. But on removing the positivity condition for a layer near the box, the order was same as in \cite{bolthausen1995}. It was also stated in \cite{BDG01} that the probability of positivity in case of GFF in a  box of dimension $2$ decays exponentially, and this is really a boundary phenomenon. So in order to look into the long range correlations, and local fluctuations, the boundary effect has to be removed. This approach has been taken in \cite{BDG01} to look into the behaviour of a typical vertex when pressed against this hard wall for a GFF.

Studies on GFF in a box of size $n$ in dimension $2$ since Bolthausen et. al. in \cite{BDG01} have utilised the covariance of GFF in the interior of the box.  The connection between the covariance structure of 2D-GFF and BRW was made in \cite{BZ10} to show tightness for the maximum of the GFF. It has been observed to be log-correlated. To further refine the results on entropic repulsion of the GFF in dimension $2$ it is imperative to consider a similar behaviour for the BRW on a tree. Our calculations heavily rely on the tail behaviour of BRW, as shown in Section \ref{sec:brwlefttail}. The connection between the tail behaviours of 2D-GFF and BRW have already been mentioned above. The multi-scale analysis, hinting towards the tree structure is made use of extensively to study the extremal properties of GFF in dimension $2$, as shown in \cite{BZ10}, \cite{BDZ13}, \cite{BDZ14}. Similar strategies have been applied to study the entropic repulsion of Gaussian membrane model for the critical dimension $4$ in \cite{Kurt09}. It has been worked out in \cite{R16}, \cite{Schweiger} that the Gaussian membrane model in the critical dimension is log-correlated. The works of \cite{DRZ15} further exhibit strong relations between the BRW and log-correlated Gaussian fields, the branching number varying according to the dimension of the box. 

In the backdrop of these studies, we consider the behaviour at a typical vertex of branching random walk on a $d-$ary tree. This is specially relevant keeping in mind the covariance structure of the BRW and that of the GFF in dimension 2, in the interior.

Entropic repulsion in case of GFF on Sierpinski carpet graphs has been covered in \cite{ChenU13}. More recently entropic repulsion in $|\nabla \phi|^p$ surfaces has been considered in \cite{CMT17}.  

We are interested in  $\P ( \phi_v^n \geq 0 ~ \forall v \in T_n)$ as well as $\E( \phi_u^n \mid \phi_v^n \ge 0 ~ \forall v \in T_n )$. We are essentially interested in the conditional distribution of BRW at a typical vertex under the condition of positivity at all vertices of level $n$. The computation of the expectation is the first step in that direction.

In regard to the behaviour of the branching random walk in presence of a hard wall, we recall similar results for other Gaussian processes such as \cite{Deuschel99}, \cite{Deuschel2000}, \cite{BDG01}, \cite{Kurt09}, \cite{ChenU13}, \cite{CMT17}. The leading order term in the exponent of the probability of positivity is what is estimated, while we estimate both the leading order term and the second leading term in the exponent. This also helps us in finding the second order term in the expected value of a typical vertex, under the hard wall condition. 

We know from \cite[Theorem 4]{ofernotes} that $\E(\max_{v \in T_n} {\phi}_v)$ is of the form $c_1 n - c_2 \log n + O(1)$. We define $m_n =c_1 n - c_2 \log n$, and $\sigma_{d,n}^2 = \frac{1-d^{-n}}{d-1}$. In fact we have explicit values of $c_1$ and $c_2$ as $\sqrt{2 \log d}$ and $\dfrac{3}{2\sqrt{2 \log d}}$ respectively. 

Our main result of this paper, in regard to the probability of positivity, is the following:

\begin{theorem}[Positivity probability]\label{thm-positivityBRW}
	There exists $\lambda'=\frac{\sqrt{2}\log n}{\sqrt{\log d}}+O(1)$, such that for $n$ sufficiently large we have,  for $K_1, K_2, K_3 > 0$ independent of $n$, and $K_4 = \frac{1}{c\sigma^2_{d,n}\log d}$, 
	\begin{equation}\label{eq-positivity} 
	K_1 e^{-\frac{1}{2\sigma^2_{d,n}}(m_n - \lambda')^2- K_3(m_{n} - \lambda') }\le \P( \phi_v^n \geq 0~ \forall v \in T_n ) \le K_2 e^{-\frac{1}{2\sigma^2_{d,n}}(m_{n} - \lambda')^2 - K_4 (m_{n} - \lambda').}
	\end{equation}
\end{theorem}

In \cite{BDG01} it has been shown that the conditional expectation under positivity is roughly close to the expected maximum for the discrete GFF in 2 dimensions. Similarly in \cite{Kurt09} for the membrane model in dimension $d=4$ a lower bound on the conditional expectation of a typical vertex, under positivity is computed to be close to the expected maximum. Here, however we show that for a branching random walk the conditional expectation is at least a constant times $\log n$ less than the expected maximum. The second main result of this paper is: 
\begin{theorem}[Expected value]\label{thm-typicalpos} We have for  $u \in T_n$, and $n$ sufficiently large enough, 
	$$m_n - \frac{3\sqrt{2}}{\sqrt{\log d}} \log n+ O(1) \le \E\left( \phi_u^n \mid \phi_v^n \ge 0 ~ \forall v \in T_n \right) \le m_n - \frac{\sqrt{2}}{\sqrt{\log d}}\log n+ O(1). $$ 
\end{theorem}
The approach that we take for proving this is that we raise the average value of the Gaussian process and then multiply a compensation probability to that. We optimise this average value so as to maximise the probability of positivity. The value at which this probability is maximised should ideally be the required conditional expectation. 

In order to prove this in details, we invoke a model called the switching sign branching random walk, which is similar in structure to the original branching random walk. The model is motivated by a similar model that had been introduced before for BRW on the lattice $\mathbb{Z}^2$ in \cite{DG15}, which was effectively a construction on $4-$ary tree. We have done a more general construction of this on a $d-$ary tree in Section \ref{sec:ssbrwdef}. We begin our calculations with a preliminary upper bound on the left tail of the maximum of the BRW in Section \ref{sec:brwlefttail}. Section \ref{sec:ssbrwdef} contains the definition of the new model switching sign branching random walk followed by a comparison of positivity for the branching random walk with this model using Slepian's lemma. A left tail computation for the maximum of this model gives us the  ingredients for proof of Theorem \ref{thm-positivityBRW}, which is in the concluding part of Section \ref{sec:ssbrwtails}. Section \ref{sec:condexpproof} contains the proof of Theorem \ref{thm-typicalpos}. The upper bound follows from Section \ref{sec:ssbrwdef}, while for the lower bound we further have to invoke the Bayes' rule and tail estimates to arrive at our result.  

\textbf{Notation:} We denote the event $\{ \phi_v^n \ge 0 ~ \forall v \in T_n\}$ by {\bf $\Lambda^+_n$}. We also term the sum of all the Gaussian variables at the level $n$ as $S_n$. In mathematical terms $S_n = \sum_{v : v \in T_n} \phi_v^n,$ where the sum contains $d^n$ terms.

\begin{remark} The representation of the BRW as a sum of two Gaussian fields, in the setting of entropic repulsion, is a key point of the article. The constant part which represents the typical value of the field helps in obtaining the height under the entropic repulsion, while the covariance fluctuations remain restored in the other part. This representation helps in optimizing over the set of possible values for the typical height of the field under positivity. 
	
\end{remark}
\begin{remark}
	Future directions along the line of this work include firstly the distributional behaviour and convergence of the branching random walk under positivity. In \cite{Deuschel99} it has been shown that the infinite GFF for $d\ge 3$ under positivity, on removing the conditioned height, converges weakly to the lattice free field. Whether a similar phenomenon can be observed in case of BRW is something that can be considered.  
\end{remark}

\begin{remark}
	Furthering our work, we can also consider the similar phenomenon for general log-correlated Gaussian fields. The splitting of the covariance matrix into two parts, one involving a constant Gaussian field, is not immediate in case of log-correlated Gaussian fields as in the form considered in \cite{DRZ15}. 
\end{remark}
\section{Left tail of maximum of BRW}\label{sec:brwlefttail} 
This section is dedicated to proving an exponential upper bound on the left tail of the maximum of a BRW. 
We first begin with a comparison lemma by Slepian for Gaussian processes in \cite{slepian62}.
\begin{lemma}\label{lem-slepian}
	Let $\mathcal{A}$ be an arbitrary finite index set and let $\{ X_a : a \in \mathcal{A}\}$ and $\{ Y_a : a \in \mathcal{A}\}$ be two centered Gaussian processes such that: $\E(X_a - X_b)^2 \ge \E(Y_a - Y_b)^2$, for all $a,b \in \mathcal{A}$ and $\V(X_a) = \V (Y_a)$ for all $a \in \mathcal{A}$. Then $\P(\max_{a \in \mathcal{A}} X_a \ge \lambda) \ge \P(\max_{a \in \mathcal{A}} Y_a \ge \lambda)$ for all $\lambda \in \mathbb{R}$.
\end{lemma} 

The main result of the section is the following: 
\begin{lemma}
	There exists constants $\bar{C} , c^* >0$ such that for all $n \in \mathbb{N}$ and $0 \le \lambda \le ( n)^{2/3}$, 
	\begin{equation}\label{eq-BRWlefttailub}
	\P( \max_{v \in T_n} \phi^n_v \leqslant m_n -\lambda ) \le \bar{C} e^{-c^* \lambda}
	\end{equation}
\end{lemma}
\begin{proof}
	From \cite[Section 2.5]{ofernotes} we have tightness for $\{ \max_{v \in T_n} \phi^n_v - m_n \}_{n \in \mathbb{N}} $, where $m_n=\sqrt{2 \log d}n-\frac{3}{2\sqrt{2 \log d}}\log n$. So there exists $ \beta > 0$ such that for all $n\ge 2$, 
	\begin{equation}\label{eq-BRWsubtreeub} 
	\P (\max_{v \in T_n} \phi^n_v \geqslant m_n - \beta ) \geqslant 1/2. 
	\end{equation} 
	Further, we also have that for some $\kappa > 0$  and for all $ n \geq n' \geq 2$ 
	\begin{equation}\label{eq-diffexpmeanbrw}
	\sqrt{2 \log d}(n-n') - \frac{3}{2 \sqrt{2 \log d}} \log( n/n') - \kappa \leqslant m_n - m_{n'} \leqslant \sqrt{2 \log d}(n-n') + \kappa. 
	\end{equation}
	Now we fix $\lambda' = \lambda/2$, and $n'=\lceil n - \frac{1}{\sqrt{2 \log d}} (\lambda' - \beta - \kappa)\rceil$, where $\lambda'$ satisfies $\lambda' \geqslant \beta + \kappa+\sqrt{2 \log d}$. From (\ref{eq-diffexpmeanbrw}) it follows then that $m_n - m_{n'} \leqslant \lambda' - \beta$. We consider a tree of height $n$ rooted at $0$. We consider all subtrees rooted at vertices $v \in T^n$ such that $d_T(0,v)=n-n'$. They are individually trees of height $n'$. 
	The total number of such subtrees we have is $d^{n-n'}$. We call their leaf nodes $\{ T_{n'}^{(1)}, T_{n'}^{(2)}, \ldots, T_{n'}^{(d^{n-n'})} \}$. Now for all $v \in T_n$, we define 
	$$\bar{\phi}^n_v = g_v^{n'}+ \phi, $$ where $(g_v^{n'})_{v \in T_n}$ are the BRWs obtained by adding the Gaussians for the edges only in the subtrees of height $n' $, and  $\phi$ is an independent Gaussian of mean $0$ and variance $n-n'$. Clearly 
	$$\V \phi_v^n = \V \bar{\phi}_v^n \quad \text{and } \quad \E \phi_v^n \phi_u^n \leq \E \bar{\phi}_v^n \bar{\phi}_u^n ~ \forall u \neq v \in T_n.$$
	So by Lemma \ref{lem-slepian}, we have 
	\begin{equation}\label{eq-brwtailcomp}
	\P(\max_{v \in T_n} \phi_v^n \leq t) \le \P(\max_{v \in T_n} \bar{\phi}_v^n \leq t) ~ \forall ~ t \in \mathbb{R}. 
	\end{equation}
	Using (\ref{eq-BRWsubtreeub}) and (\ref{eq-diffexpmeanbrw}), one has for all $i \in \{1, 2, \ldots, d^{n-n'} \}$, 
	\begin{eqnarray*}
		\P(\sup_{v \in T_{n'}^{(i)}} g_v^{n'} \geq m_n - \lambda' ) &=& \P(\sup_{v \in T_{n'}^{(i)}} g_v^{n'} \geq m_{n'}+ m_n - m_{n'}  - \lambda' ) \\
		&\ge& \P(\sup_{v \in T_{n'}^{(i)}} g_v^{n'} \geq m_{n'} - \beta )  \ge 1/2
	\end{eqnarray*} 
	and so $\P(\sup_{v \in T_n} g_v^{n'} < m_n - \lambda' ) \leq (\frac{1}{2})^{d^{n-n'}}$. 
	
	Therefore, 
	$$\P(\sup_{v \in T_n} \bar{\phi}_v^{n} \leq m_n - \lambda ) \leq \P(\sup_{v \in T_n} g_v^{n'} < m_n - \lambda' ) + \P(\phi \leq -\lambda' ) \leq  \bar{C} e^{-c^* \lambda}, $$ 
	for some $\bar{C} , c^* > 0$. Now in conjunction with (\ref{eq-brwtailcomp}), the lemma is proved. Note that we choose $\bar{C} , c^* > 0$ such that the inequality holds for $\lambda \geqslant 2(\beta +\kappa+\sqrt{2 \log d})$ as well as $\lambda < 2(\beta +\kappa+\sqrt{2 \log d})$.  
\end{proof}
\section{Switching Sign Branching Random Walk}\label{sec:ssbrwdef}
At this juncture we define a new Gaussian process on the tree, which we call the switching sign branching random walk. This was used to approximate the branching random walk in \cite{DG15} in case of a $4$-ary tree. We have generalised the process for a $d$-ary tree. The switching sign branching random walk consists of two parts, one that varies across vertices, and the other that is fixed over vertices. The first part of the process, which is not fixed over vertices, is different from the normal branching random walk in the sense that instead of the $d$ edges coming out of it being associated to independent normal random variables, they are associated with linear combinations of $d-1$ independent Gaussians, such that the covariance between any two of them is the same, and all of them add up to zero. The existence of this is guaranteed by the following Lemma. 
\begin{lemma}\label{lem-nodessbrw}
	There exists $A \in \mathbb{R}^{(d-1)\times(d-1)}$ such that for $W \sim N(0, \sigma^2 I_{(d-1) \times (d-1)})$, the covariance matrix of $\tilde{Y}=(Y_1, Y_2, \cdots, Y_{d-1})^{\top}=AW$ has diagonal entries equal to $\sigma^2$ and all its off-diagonal entries equal (say $\eta$). Further $\V(\1^{\top} AW)=\sigma^2$ and $\Cov(-\1^{\top}AW, (AW)_i)=\eta$ for all $i \in \{1, 2, \ldots, d-1\}$. Here by $\1$ we represent the column vector of size $d-1$ with all its entries as $1$, and $\1^{\top}$ is the transpose of it. 
\end{lemma}
\begin{proof}
	We know that the covariance matrix for $\tilde{Y}$ is $\sigma^2 AA^{\top}$. Further from the condition that $\V(\1^{\top} AW)=\sigma^2$ we get that $\eta=-\frac{\sigma^2}{d-1}$. So in order for $A$ to exist we must have 
	\[
	AA^{\top}=  \begin{bmatrix}
	1       & -\frac{1}{d-1}  & -\frac{1}{d-1} & \dots & -\frac{1}{d-1} \\
	-\frac{1}{d-1}       & 1 & -\frac{1}{d-1} & \dots & -\frac{1}{d-1} \\
	\vdots & \vdots & \vdots & \ddots & \vdots \\
	-\frac{1}{d-1}      & -\frac{1}{d-1} & -\frac{1}{d-1} & \dots & 1
	\end{bmatrix}_{(d-1) \times (d-1)}.
	\]
	Since the matrix on the right hand side is a symmetric matrix with non-negative eigenvalues, by Cholesky decomposition we obtain the existence of such an $A$. In particular, we have a specific choice of $A=\sqrt{\frac{d}{d-1}}I_{(d-1)\times(d-1)}-\frac{\sqrt{d}-1}{(d-1)^{3/2}}J_{(d-1)\times(d-1)}$, where $J_{k\times k}$ is the square matrix of order $k$ with all entries equal to $1$.
\end{proof}
We now define $Y_d=-\1^{\top}\tilde{Y}$, and  $Y=(Y_1, Y_2, \cdots, Y_{d})^{\top}$. This $Y$ is a degenerate multivariate normal, with variance-covariance matrix given as follows: 
\begin{equation}\label{eq-ssbrw-definition}
\begin{split}
\V (Y_i) &= \sigma^2 \quad \text{for all } i \in \{1,2,\ldots, d\} \\
\Cov(Y_i, Y_j) &=  - \frac{\sigma^2}{d-1} \quad \text{ for all }i \neq j \in \{1,2,\ldots, d\} \,.
\end{split}
\end{equation}
A pictorial representation of a node for the SSBRW process is given in Figure \ref{fig-ssbrw}. 
\begin{figure}[ht]
	\begin{center}
		\tikz[scale=1.25]{
			\draw[black, thick] (5,4) -- (1,2);
			\draw[black, thick] (5,4) -- (3,2);
			\draw[black, dashed] (5,4) -- (4.8,2);
			\draw[black, dashed] (5,4) -- (5.2,2);
			\draw[black, thick] (5,4) -- (7,2);
			\draw[black, thick] (5,4) -- (9,2);
			\draw[black, thick, fill=white] (5,4) circle[radius=.25cm];	
			\draw[black, thick, fill=white] (1,2) circle[radius=.35cm];
			\draw[black, thick, fill=white] (3,2) circle[radius=.35cm];	
			\draw[black, thick, fill=white] (7,2) circle[radius=.35cm];				\draw[black, thick, fill=white] (9,2) circle[radius=.35cm];	
			
			\node[ann] at (5,4) {$Z$};
			\node[ann,scale=.7] at (1,2) {$Z+Y_1$};
			\node[ann,scale=.7] at (3,2) {$Z+Y_2$};
			\node[ann,scale=.7] at (7,2) {$Z+Y_{d-1}$};
			\node[ann,scale=.7] at (9,2) {$Z+Y_d$};
			\node[ann,scale=.7] at (3,3 ) {$Y_1=(AW)_1$};
			\node[ann,scale=.7] at (4.2, 3) {$Y_2=(AW)_2$};				
			\node[ann,scale=0.7] at (6, 3) {$Y_{d-1}=(AW)_{d-1}$};
			\node[ann,scale=0.7] at (7.8, 3) {$Y_d= - \sum_{i=1}^{d-1} (AW)_i$};	
			
		}
		\caption{Node of the varying part of SSBRW}
		\label{fig-ssbrw}
	\end{center}
\end{figure}

We now provide a few heuristic descriptions of the SSBRW, followed by formal defintion. 	

We consider a particle starting from a random point $X \in \mathbb{R}$, which dies at time $1$, and splits into $d$ number of children. The joint distribution of the distances travelled by these children is given by (\ref{eq-ssbrw-definition}), with choice of $\sigma^2=1-d^{-n}$.  Then at time $2$, each of these children die, and give rise to $d$ number of children each, which in turn follow the same process, the displacements at each step being independent of the displacements in the previous time points, except for the variances and covariances at level $l$ being $1-d^{-(n-l+1)}$ and $-\frac{1-d^{-(n-l+1)}}{d-1}$ each respectively. At time $n$ we have $d^n$ many particles, each having a displacement. The displacements of the collection of these particles is called a switching sign branching random walk at time $n$. The random term $X$ is normally distributed, with mean $0$ and variance $\frac{1- d^{-n}}{d-1}$, and is independent of everything else.

We can equivalently define a  switching sign branching random walk as a Gaussian process on $T^n$. The root is fixed at a random point $X$, and the displacements at each step are attached to the edges. So, the first $d$ displacements are attached to the edges between the root and its $d$ children. To each vertex we attach a quantity equal to sum of the Gaussians that we encounter while looking at the shortest path between itself and the root, plus the random variable attached to the root. Then, the collection of the displacements of the $d^n$ particles at time $n$, is given by the Gaussians attached to all the vertices in $T_n$. We denote it by $\{ \xi_v^n : v \in T_n \}$. This is the switching sign branching random walk at time $n$. Two particles at time $n$ having risen from the same ancestor at time $k$ $(k \leq n)$, is equivalent to two leaf nodes, which have branched out from the same vertex at level $k$ of the tree. Each of these displacements are sums of $n$ independent Gaussians, plus the random term $X$ which is independent of everything else.

The formal definition of the SSBRW is now provided below:
\begin{defn}
	The collection $\{ Y_{i,j}: j=1,2,\ldots, d^i, i =1,2,\ldots, n \}$ represents displacements, and are Gaussian random variables. But unlike the BRW, we have $\V(Y_{i,j})=1-d^{-(n-i+1)}$. Also, the collection of $\{Y_{i,j} : j=1,2,\ldots, d^i, i =1,2,\ldots, n \}$ does not represent a collection of independent random variables. Rather we have $\sum_{j'=1}^d Y_{i,md+j'}=0$, for all $m=\{0,1,\ldots, d^{i-1}-1\},$ $ i=\{1,2,\ldots, n\}$. We do however still have that $Y_{i_1,j_1}$ and $Y_{i_2, j_2}$ are independent if $i_1\neq i_2$. Also, $Y_{i,j_1}$ and $Y_{i, j_2}$ are independent if $\lceil j_1/d \rceil \neq \lceil j_2/d \rceil$. Otherwise, if $\lceil j_1/d \rceil =\lceil j_2/d \rceil$ then $Y_{i,j_1}$ and $Y_{i, j_2}$ are not independent, and $\Cov(Y_{i,j_1}, Y_{i,j_2})=-\frac{1-d^{-(n-i+1)}}{d-1}$. There are $d^n$ many leaf nodes in the tree, and we can fix an ordering of the vertices from $1$ to $d^n$. For any $v \in T_n$, i.e. $v \in \{1,2,\ldots, d^n \}$, we define $a_i(v)=\lceil \frac{v}{d^{n-i}}\rceil$ for $i=1,2,\ldots, n$. Then we can define $\tilde{\phi}_v^n=\sum_{j=1}^n Y_{j,a_j(v)}$. The switching sign branching random walk is given by
	\begin{equation}\label{eq-ssbrwdefn}
	\xi^n_v = \tilde{\phi}_v^n + X
	\end{equation}
	where $X$ is an independent Gaussian variable with mean zero and variance $\frac{1- d^{-n}}{d-1}$.
\end{defn}

Figure \ref{SSBRWbinary} gives a pictorial representation of the switching sign branching random walk for the case $d=2$, i.e. on a binary tree. All the $Y_{i,j}$'s represent displacements, and are distributed as Gaussian. In this figure we show the dependence described in the paragraph above by replacing one of the dependent random variables by the relevant function of the others which it correlates with. 

In this construction, unlike the BRW, we have a different variance for each level $l$ ($1\le l \le n$). Here, level $1$ denotes the edge connecting the root to its children and level $n$ denotes the edges joining the leaf nodes to their parents. 
We denote this switching sign branching random walk on the leaf nodes $T_n$ as $\{ \xi_v^n : v \in T_n \}$. For $v \in T_n$ we denote the Gaussian variable that is added on level $l$, on the path connecting $v$ to the root, by $Y_{l,a_l(v)}$ as defined before. We have $\V( Y_{l,a_l(v)})= 1 - d^{-(n-l+1)}$ as stated in the formal definition. The switching sign branching random walk will consist of two parts, the first coming from the contribution at different levels in the tree, which is $\tilde{\phi}_v^n \overset{def}{=} \sum_{l=1}^n Y_{l,a_l(v)}$. The second part is the random variable $X$, which is an independent Gaussian variable with mean zero and variance $\frac{1- d^{-n}}{d-1}$. The two parts are summed together to get $\xi_v^n$.

\begin{figure}[h!]\label{SSBRWbinary}
	\begin{center}
		\tikz[scale=.8]{
			\draw[black, thick] (8,8) -- (6,6);
			\draw[black, thick] (6,6) -- (4.5,4);
			\draw[black, thick] (6,6) -- (7.5,4);
			\draw[black, dashed] (4.5,4) -- (3.5,2);			
			\draw[black, dashed] (4.5,4) -- (5.5,2);
			\draw[black, thick] (3.5,2) -- (4.5,0);
			\draw[black, thick] (3.5,2) -- (2.5,0);
			\draw[black, dashed] (6.5,4) -- (6.5,0);			
			\draw[black, dashed] (8,4) -- (8,0);
			\draw[black, dashed] (9.5,4) -- (9.5,0);
			\draw[black, dashed] (11.5,4) -- (12.5,2);			
			\draw[black, dashed] (11.5,4) -- (10.5,2);
			\draw[black, thick] (10,6) -- (8.5,4);
			\draw[black, thick] (10,6) -- (11.5,4);
			\draw[black, thick] (8,8) -- (10,6);
			\draw[black, thick] (12.5,2) -- (11.5,0);
			\draw[black, thick] (12.5,2) -- (13.5,0);
			\draw[black, thick, fill=white] (8,8) circle[radius=.45cm];				
			\draw[black, thick, fill=white] (6,6) circle[radius=.45cm];			
			\draw[black, thick, fill=white] (10,6) circle[radius=.45cm];
			\draw[black, thick, fill=white] (4.5,4) circle[radius=.45cm];
			\draw[black, thick, fill=white] (4.5,0) circle[radius=.45cm];				
			\draw[black, thick, fill=white] (2.5,0) circle[radius=.45cm];	
			\draw[black, thick, fill=white] (13.5,0) circle[radius=.45cm];				
			\draw[black, thick, fill=white] (11.5,0) circle[radius=.45cm];
			\node[ann] at (8,8 ) {$X$};	
			\node[ann] at (7,7 ) {$Y_{1,1}$};
			\node[ann, scale=.6] at (6,6 ) {$X+Y_{1,1}$};
			\node[ann,scale=.4] at  (4.5,4)  {$X+ Y_{1,1}+Y_{2,1}$};
			\node[ann, scale=.6] at (10,6 ) {$X-Y_{1,1}$};
			\node[ann] at (5,5) {$Y_{2,1}$};
			\node[ann] at (7,5) {$-Y_{2,1}$};
			\node[ann] at (9, 5) {$Y_{2,3}$};	
			\node[ann] at (11, 5) {$-Y_{2,3}$};	
			\node[ann] at (3, 1) {$Y_{n,1}$};	
			\node[ann] at (4.3, 1) {$-Y_{n,1}$};
			\node[ann,scale=.7] at (13.3, 1) {$-Y_{n,2^n-1}$};	
			\node[ann,scale=.7] at (12, 1) {$Y_{n,2^n-1}$};
			\node[ann] at (2.5, -.65) {$1$};
			\node[ann] at (4.5, -.65) {$2$};
			\node[ann] at (11.5, -.65) {$2^n-1$};
			\node[ann] at (13.5, -.65) {$2^n$};
			\node[ann, scale=.7] at (2.5, 0) {$X+\tilde{\phi}^n_{1}$};	
			\node[ann, scale=.7] at (4.5, 0) {$X+\tilde{\phi}^n_{2}$};
			\node[ann, scale=.7] at (11.5, 0) {$X+\tilde{\phi}^n_{2^n-1}$};
			\node[ann, scale=.7] at (13.5, 0) {$X+\tilde{\phi}^n_{2^n}$};			
			\node[ann,scale=.65] at (0.7, 0) {$\tilde{\phi}^n_{1}=X+\sum_{j=1}^{n}Y_{j,1}$};			
			\node[ann] at (9, 7) {$-Y_{1,1}$};		
		}
	\end{center}
	\caption{SSBRW on Binary tree }
\end{figure}

The covariance structure for this new model closely resembles that of the branching random walk. The following lemma deals with this comparison: 
\begin{lemma}\label{lem-compbrwssbrw}
	The Gaussian fields $\{ \xi^n_v: v \in T_n \}$ and $\{ \phi_v^n : v \in T_n \}$ are identically distributed. 
\end{lemma}
\begin{proof}
	First we show that the variances are identical for the two processes. To this end, we begin by computing the variance of $\xi^n_v$ as follows: 
	\begin{eqnarray*}
		\V(\xi^n_v) &=& 1 - d^{-1}+ 1 - d^{-2}+ \cdots + 1 - d^{-n} + \frac{1- d^{-n}}{d-1}\\
		&=& n - \frac{1- d^{-n}}{d-1}+ \frac{1- d^{-n}}{d-1}= n.
	\end{eqnarray*}
	Next, in case of the covariances we consider $u, v \in T_n, u \neq v$, such that they have the last common ancestor at generation $n-k$, i.e.\ $\Cov(\phi^n_u, \phi^n_v)=n-k$. Then we have 
		\begin{equation}\label{eq:ssbrwcov}
		\begin{array}{l}
		a_i (u) = a_i (v) ~ \text{for } 1\leq i \leq n- k, \\
		a_i (u) \neq a_i (v) ~ \text{for } n-k+1\leq i \leq n, \\
		\lceil a_{n-k+1}(u)/d \rceil = \lceil a_{n-k+1}(v)/d \rceil, \\
		\lceil a_{j}(u)/d \rceil \neq \lceil a_{j}(v)/d \rceil ~ \text{for } n-k+2\leq i \leq n.
		\end{array}
		\end{equation}
		We know that $\Cov( \tilde{\phi}_u^n, \tilde{\phi}_v^n)= \sum_{i=1}^n \Cov(Y_{i, a_i(u)}, Y_{i, a_i(v)})$ from the fact that $Y_{i_1,j_1}$ and $Y_{i_2, j_2}$ are independent if $i_1\neq i_2$. 
		From (\ref{eq:ssbrwcov}) we have, 
		\begin{equation*}
		\Cov( \tilde{\phi}_u^n, \tilde{\phi}_v^n)= \sum_{i=1}^{n-k} \V(Y_{i, a_i(u)})+ \Cov(Y_{n-k+1, a_{n-k+1}(u)}, Y_{n-k+1, a_{n-k+1}(v)}).
		\end{equation*}
		Plugging in the values we get, 
		$$\Cov( \tilde{\phi}_u^n, \tilde{\phi}_v^n)= - \frac{1- d^{-k}}{d-1} + \sum_{l=k+1}^{n}(1- d^{-l})= n - k - \frac{1- d^{-n}}{d-1}.$$
		Hence, $\Cov(\xi^n_u, \xi^n_v)=\V(X) + \Cov( \tilde{\phi}_u^n, \tilde{\phi}_v^n) = n-k $. 
	So, the covariance structures for the fields $\xi$ and $\phi$ match, and hence they are identically distributed. 
\end{proof}
A simple corollary of Lemma \ref{lem-compbrwssbrw}, is the following, based on the fact that the two processes have identical distributions. 
\begin{cor}
	We have the following equality: 
	\begin{equation}\label{eq-poscomp}
	\P( \phi_v^n \geq 0~ \forall v \in T_n ) = \P( \max_{v \in T_{n}} \tilde{\phi}_v^{n} \le X ) 
	\end{equation} 
\end{cor}
\begin{cor}
	From \cite[Theorem 4]{ofernotes}, we have $\E \max_{v \in T_n} \phi_v^n = n \sqrt{2\log d } - \frac{3}{2 \sqrt{2 \log d}}\log n + O(1)$. Therefore, 
	\begin{equation*}
	\E \max_{v \in T_n } \tilde{\phi}^n_v = n \sqrt{2 \log d } - \frac{3 \log n}{2 \sqrt{2 \log d}} + O(1).
	\end{equation*}
\end{cor}
\begin{cor}
	There exists constants $\bar{C}' , c^* >0$ such that for all $n \in \mathbb{N}$ and $0 \le \lambda \le ( n)^{2/3}$, 
	\begin{equation}\label{eq-ssBRWlefttailub}
	\P( \max_{v \in T_n} \tilde{\phi}^n_v \leqslant m_n -\lambda ) \le \bar{C}' e^{-c^* \lambda}
	\end{equation}
\end{cor}
\begin{proof}
	$$\frac{1}{2}\P( \max_{v \in T_n} \tilde{\phi}^n_v \leqslant m_n -\lambda ) = \P( \max_{v \in T_n} \tilde{\phi}^n_v \leqslant m_n -\lambda, X \leq 0  ) \leq \P( \max_{v \in T_n} \phi^n_v \leqslant m_n -\lambda ).$$
	Now using (\ref{eq-BRWlefttailub}), and with $\bar{C}'= 2\bar{C}$ we arrive at (\ref{eq-ssBRWlefttailub}). 
\end{proof}

\section{Estimates on left tail and positivity}\label{sec:ssbrwtails} 

From the equation (\ref{eq-poscomp}) we understand that the probability of positivity for the branching random walk can be computed using bounds on the left tail of the maximum of $\tilde{\phi}^n_.$, a part of the switching sign branching random walk, as the left tail is heavily concentrated around the maximum. This motivates the following computations on the left tail of the maximum. 
\begin{lemma}\label{lem-lefttail}
	We call $c=1/\sqrt{2 \log d}$(where $m_n = \sqrt{2 \log d} n - \frac{3}{2\sqrt{2 \log d}} \log n$) to be the constant such that $|m_{n-c\lambda }- m_n+\lambda|\rightarrow 0$ as $n\rightarrow \infty$, where $\lambda=\lambda(n)=o(n)$ is positive. Then there exists constants $C', C'', K', K''$ independent of $n$ such that for sufficiently large $n$ we have:
	\begin{equation}\label{eq-lefttailssbrw}
	K' \exp(-K''d^{c\lambda})  \le \P(\max_{v \in T_n} \tilde{\phi}_v \le m_n - \lambda ) \leq C'\exp(-C''d^{c\lambda}) .
	\end{equation} 
\end{lemma} 

\begin{proof}
	We work with $\P(\max_{v \in T_n} \tilde{\phi}_v \le m_{n - c\lambda})$ as due to our definition of $c$, for sufficiently large $n$ this probability is close to $\P(\max_{v \in T_n} \tilde{\phi}_v \le m_n - \lambda )$. From \cite{BDZ14,DRZ15}, we can see that $\{ \max_{v \in T_n} \tilde{\phi}_v - m_n \}$ converges in distribution, as it is equivalent in distribution to a BRW, after adding the same independent Gaussian to all points. Hence  $\P(\max_{v \in T_n} \tilde{\phi}_v \le m_{n - c\lambda})$ and $\P(\max_{v \in T_n} \tilde{\phi}_v \le m_n - \lambda )$ converge to the same point.  We know that the BRW is a Gaussian field which is obtained by adding the same Gaussian to all vertices of an SSBRW. This helps us find bounds on lower and upper tails of maximum of SSBRW using results on convergence of maximum of BRW, as proved in \cite{aidekon}, \cite{BDZ14} etc. 
	
	
	We also force $\lambda$ to be such that $c\lambda$ is an integer. For other values of $\lambda$, we can adjust for the constants by looking at $\lceil c\lambda \rceil$ and $\lfloor c \lambda \rfloor$.
	We first consider the tree only up to the level $c\lambda$ and consider the cumulative sum of the Gaussian variables at these vertices till the level $c\lambda$. 
	We rename all these Gaussian variables at level $c\lambda$ of this new tree to be $A_1, A_2, \ldots, A_{d^{c\lambda}}$. Then each $A_i$ is essentially of the form $\sum_{j=1}^{c\lambda} Y_{j,a_j(v)}$ for some $v\in T_n$. We know that the definition in Section \ref{sec:ssbrwdef} of switching sign branching random walk model guarantees $\sum_{i=1}^{d^{c\lambda}} A_i = 0$. We consider the subtrees rooted at the vertex which has values $A_i$ and call its maximum to be $M_i$. These are trees of height $n-c\lambda$ and hence we have $\E M_i = m_{n - c \lambda} + O(1) ~~ \forall i$ and $M := \max_{v \in T_n} \tilde{\phi}_v= \max_{i=1}^{d^{c\lambda}} (M_i+A_i)$. We want to obtain bounds for the probability $\P(\max_{v \in T_n} \tilde{\phi}_v \leq m_{n-c\lambda})$. We condition on the values of $A_1, A_2, \ldots, A_{d^{c\lambda}}$ which in turn breaks down the required probability in a product form since the maxima for the $d^{c \lambda}$ subtrees are independent and have identical distributions. We have the following,   
	\begin{align*}
	\P(\max_{v \in T_n} \tilde{\phi}_v \leq m_{n-c\lambda} \mid A_1, A_2, \ldots, A_{d^{c\lambda}}) 
	= \P(\max_{i=1}^{d^{c\lambda}} (M_i + A_i) \leq m_{n-c\lambda} \mid A_1, A_2, \ldots, A_{d^{c\lambda}} ). 
	\end{align*}
	This can be further broken down from independence as, 
	\begin{align}\label{eq-ssbrw-brk}
	\P(\max_{i=1}^{d^{c\lambda}} (M_i + A_i) \leq m_{n-c\lambda} \mid A_1, A_2, \ldots, A_{d^{c\lambda}} )=\Prod_{i=1}^{d^{c\lambda}} \P(M_i + A_i \leq m_{n-c\lambda} \mid A_i ).\end{align}
	
	The right hand side of (\ref{eq-ssbrw-brk}) satisfies the following inequality: 
	$$\Prod_{i=1}^{d^{c\lambda}} \P(M_i + A_i \leq m_{n-c\lambda} \mid A_i ) \leq \Prod_{i: A_i > 0}^{d^{c\lambda}} \P(M_i \leq m_{n-c\lambda}-A_i \mid A_i ),$$
	which happens since for the cases where $A_i < 0$ we bound the terms in the product by $1$. The right hand side of the last inequality is further bounded by: 
	\begin{equation}\label{eq-ssbrw-ltbd}
	(\bar{C'}\vee1)^{d^{c\lambda}}  \exp(-c^* \Sum_{i=1}^{d^{c\lambda}}A_i^+) =   \exp({d^{c\lambda}} \log (\bar{C'}\vee1) -c^* \Sum_{i=1}^{d^{c\lambda}}A_i^-) 
	\end{equation}
	In the final two steps we first make use of (\ref{eq-ssBRWlefttailub}), followed by the fact that $\sum_i A_i =0$. 
	We consider two different cases:
	\begin{itemize}
		\item[1)] When $A_i^- \leq 2\bar{A}$ for at least $d^{c\lambda}/2$ many $i$, where $\bar{A}$ is a positive constant to be chosen later on.
		\item[2)] When 1) doesn't happen and so then $\sum_{i=1}^{d^{c\lambda}} A_i^- \geq \bar{A} d^{c\lambda}$.
	\end{itemize}
	Even for the first case we break it down into two parts according to whether $\sum_{i=1}^{d^{c\lambda}} A_i^- \geq \bar{A} d^{c\lambda}$ or not. 
	
	
	When 2) holds then clearly (\ref{eq-ssbrw-ltbd}) is bounded by $\exp(- (c^* \bar{A} - \log (\bar{C'}\vee1)) d^{c\lambda})$, and now on choosing $\bar{A}$ such that $c^* \bar{A} - \log (\bar{C'}\vee1)> 0$ we have $c^{**}>0$ such that our required term is bounded by $\exp(- c^{**} d^{c\lambda})$. 
	
	In the other case also $$\P(M_i \leq m_{n-c\lambda}-A_i \mid A_i ) \leq \P(M_i \leq m_{n-c\lambda}+2\bar{A})$$ for those $i$ for which $A_i^- \leq 2\bar{A}$. From the lower bound on the right tail of the maximum of a branching random walk (see \cite[eq. (2.5.11)]{ofernotes}), we can find $p$, independent of $n$, where $0<p<1$ such that $\P(M_i \leq m_{n-c\lambda}+2\bar{A})<p$ for all sufficiently large $n$  and so the probability $\Prod_{i=1}^{d^{c\lambda}} \P(M_i + A_i \leq m_{n-c\lambda} \mid A_i )$ is bounded by $\exp(-\bar{c}d^{c\lambda})$. Now from this $\bar{c}$ and $c^{**}$ we select one unified $C', C''$ so that 
	$$\P(\max_{v \in T_n} \tilde{\phi}_v \leq m_{n-c\lambda})\leq C'\exp(-C''d^{c\lambda}) .$$
	Again for the lower bound we have 
	
	\begin{eqnarray*}
		\P(\max_{v \in T_n} \tilde{\phi}_v \le m_{n - c\lambda} ) &=& \int_{\mathbb{R}^{d^{c\lambda}}} \Prod_{i=1}^{d^{c\lambda}} \P(M_i \leq m_{n-c\lambda}-  A_i )  dP_{A_1,\ldots, A_{d^{c\lambda}}} \\
		&\ge& (\bar{p})^{d^{c\lambda}} \int_{[-1,1]^{d^{c\lambda}}} dP_{A_1,\ldots, A_{d^{c\lambda}}}, 
	\end{eqnarray*}
	where $\bar{p}$ is chosen to be a lower bound on $\P(M_i \leq m_{n-c\lambda}-1)$ for all sufficiently large $n$, which can be obtained from using convergence results on maximum of branching random walk. 
	Now $\{A_1, A_2, \ldots, A_{d^{c\lambda}}\}$ are obtained by linear combinations of $\{Y_{i,j}: j=1,2,\ldots, d^i, i=1,2,\ldots, c\lambda \}$, each of which Gaussian random variables, each being obtained from $c\lambda$ many of them(which are also independent), and a way to make all $A_i$'s in the range $[-1,1]$ is to make absolute value of the contribution at the $j$th level, (i.e. $Y_{j,a_v(j)}$) to be bounded by $\frac{1}{10(c\lambda+1-j)^2}$, for $j=1, 2, \ldots, c\lambda$. Each of these $Y_{j,a_v(j)}$s come from translation of independent standard Gaussians, which we put bounds on. So the independent standard Gaussians for level $j$ are bounded by $\frac{1}{10 \sqrt{d}(c \lambda +1-j)^2}$.
	So this gives, for some constant $K>0$,  
	\begin{eqnarray*}
		\P(\max_{v \in T_n} \tilde{\phi}_v \le m_{n - c\lambda} )
		\ge (\bar{p})^{d^{c\lambda}} \prod_{j=1}^{c\lambda} \left( \frac{1}{10 K \sqrt{d}(c \lambda +1-j)^2} \right)^{(d-1)d^{j-1}} . 
	\end{eqnarray*}
	We take a logarithm of this term above, which leads to a sum. Approximation of the sum, as shown below in Lemma \ref{lem-sumapprox}, proves (\ref{eq-lefttailssbrw}).
\end{proof}

\begin{lemma}\label{lem-sumapprox}
	$\sum_{j=1}^{n} (\log |n+1-j|) d^j$ is of order $\Theta(d^n)$.
\end{lemma}
\begin{proof} We begin with an upper bound on the sum. We use a trivial bound of $\log  |x| \le |x|$ for $|x| \geq 1$, followed by a few series summations. 
	\begin{eqnarray*}
		\sum_{j=1}^{n} (\log |n+1-j|) d^j &\le& \sum_{j=1}^{n} ( |n+1-j|) d^j  \\
		&=& (n+1) \sum_{j=1}^{n} d^j  -  \sum_{j=1}^{n} j d^j  \\
		&=& (n+1) \frac{d^{n+1}-d}{d-1} - \frac{nd^{n+2}- (n+1)d^{n+1} +d}{(d-1)^2} \\
		&=& \frac{d^{n+2}- (n+1)d^{2} +nd}{(d-1)^2} 
	\end{eqnarray*}
	This gives an upper bound of order $d^n$. The lower bound follows easily. 
\end{proof}
We now look back into our question of the branching random walk being positive at all vertices. We know that the maximum of the BRW is heavily concentrated around the expected maximum. Using this fact, in a neighbourhood around the maximum, we further try to maximise the probability of the maximum being there. This point where this occurs will also roughly be the typical value of a vertex. This motivates the proof of Theorem \ref{thm-positivityBRW}.  
\begin{proof}[Proof of Theorem \ref{thm-positivityBRW}]
	{\it Upper bound: }	From (\ref{eq-poscomp}) we have an upper-bound on the probability of positivity based on the switching sign branching random walk. We optimise this bound by first raising the mean to a level and look at the compensation we have to apply correspondingly. We optimise over these two to obtain our bound. We apply a similar strategy for obtaining the lower bound as well. We recall (\ref{eq-poscomp}) at this juncture along with $X$, and the variance of $X$ to be $\sigma_{d,n}^2 = \frac{1-d^{-n}}{d-1}$. Let us recall the event $\Lambda_n^+$ defined before as $\{ \phi_v^n \ge 0 ~ \forall v \in T_n\}$.
	In (\ref{eq-poscomp}), we condition on the value of $X$ to obtain the following:
	$$\P(\Lambda_n^+) = \frac{1}{\sigma_{d,n}\sqrt{2\pi}}\int_{-\infty}^{\infty} \P( \max_{v \in T_{n}} \tilde{\phi}_v^{n}\le x)\exp({-x^2/2\sigma^2_{d,n}})dx$$
	Instead of integrating over $x$ we may as well replace $x$ by $m_{n} - \lambda$, and then integrate over $\lambda$. We split the integral into three parts, first with $\{ -\infty < \lambda \le 0\} $, second with $\{\frac{3}{c}\log_d n \le \lambda < \infty \}$ and the rest. From tail estimates of a Gaussian, the first part is bounded by $O(\exp(-\frac{1}{2\sigma^2_{d,n}}(m_n - \lambda')^2))$. From (\ref{eq-lefttailssbrw}), we know that the second part is bounded by $C' \exp(-C''n^3)$. The rest part has an upper bound:  
	\begin{equation}\label{eq-uboptimizer}
	\frac{C'}{\sigma_{d,n}\sqrt{2\pi}} \int_{0}^{\frac{3}{c}\log_d n } \exp(-C''d^{c\lambda})\exp({-(m_{n} - \lambda)^2/2\sigma^2_{d,n}})d\lambda. 
	\end{equation}
	We maximize the integrand in (\ref{eq-uboptimizer}), over the range of the integral, to obtain an optimal $\lambda$, say $\lambda'$, which is of order $\log n$. It satisfies the equation 
	$$m_{n}-\lambda' = \sigma^2_{d,n}C'' c d^{c\lambda'} \log d .$$
	Recalling $m_n=\sqrt{2 \log d}n - \frac{3}{2\sqrt{2 \log d}}\log n$ and $c=1/\sqrt{2 \log d}$ from Lemma \ref{lem-lefttail}, we can see that, 
		\begin{equation*}
		\dfrac{2n - \frac{3 \log n}{2 \log d} - \lambda'/\sqrt{2 \log d}}{d^{c \lambda'}} \underset{n \to \infty}{\longrightarrow} \frac{C''}{d-1}.
		\end{equation*}
	This implies that $\lambda'=\frac{\sqrt{2} \log n}{\sqrt{\log d}}+O(1)$. Plugging in we obtain an upper bound as in (\ref{eq-positivity}).   
	
	{\it Lower bound:}
	Again recalling (\ref{eq-lefttailssbrw}) we obtain that 
	{$$ \P(\Lambda_n^+) \ge \frac{K'}{\sqrt{2\pi}\sigma_{d,n}}\int_{\lambda'}^{\lambda'+1}  e^{-K''d^{c\lambda}} \exp({-(m_n - \lambda)^2/2\sigma^2_{d,n}})d\lambda .$$}
	The integrand here is infact a decreasing function of $\lambda$ in the range $\lambda \in [\lambda' , \lambda' + 1 ]$, where $\lambda'$ is from the first part of the proof. This gives a lower bound of 
	$$\frac{K'}{\sqrt{2\pi}\sigma_{d,n}}  e^{-K''d^c d^{c\lambda'}} \exp({-(m_n - \lambda' - 1)^2/2\sigma^2_{d,n}}).$$
	So, we obtain the required lower bound in (\ref{eq-positivity}). 
\end{proof}
\section{Expected value of a typical vertex under positivity}\label{sec:condexpproof}
\begin{proof}[{Proof of Theorem \ref{thm-typicalpos}}]
	We want to compute $\E\left(\frac{S_n}{d^n} \mid \Lambda^+_n \right)$. Due to Lemma \ref{lem-compbrwssbrw}, this is equivalent to computing $\E\left(\dfrac{\sum_{v=1}^{d^n} \xi_v^n}{d^n} \mid \xi^n_u \ge 0 ~ \forall ~ u \in T_n \right)= \E\left(X \mid \max_{ v \in T_n} \tilde{\phi}_v^n \le X \right)$. The conditioning events $\{ \xi^n_u \ge 0 ~ \forall ~ u \in T_n\}$ and $\{ \max_{ v \in T_n} \tilde{\phi}_v^n \le X\}$ are equivalent, since $(\tilde{\phi}_v^n)_{v \in T_n}$ is symmetric around $0$. Also,  $\frac{\sum_{v=1}^{d^n} \xi_v^n}{d^n}$ equals $X$ since $\sum_{v=1}^{d^n} \tilde{\phi}_v^n=0$. So the previous equality holds. 
	
	{\bf Upper Bound:} 
	We first split the expectation into two parts, one concerning the contribution of the right tail in the integral and the rest. We aim to show that the contribution of the right tail is negligible, thereby implying that the main contribution is from the rest, which gives an upper bound on the expectation. The tail here is motivated by the maximizer in Theorem \ref{thm-positivityBRW}. 
	\begin{eqnarray*}
		\E\left(X \mid \max_{ v \in T_n} \tilde{\phi}_v^n \le X \right) &=& \frac{1}{\sqrt{2\pi}\sigma_{d,n}}\int_{-\infty}^{\infty} x e^{-x^2/2\sigma^2_{d,n}} \dfrac{\P(\max_{v \in T_n} \tilde{\phi}_v^n \le x)}{\P(\max_{v \in T_n} \tilde{\phi}_v^n \le X)} dx \\
		&=& \frac{1}{\sqrt{2\pi}\sigma_{d,n}}\int_{-\infty}^{m_n-a\log n } x e^{-x^2/2\sigma^2_{d,n}} \dfrac{\P(\max_{v \in T_n} \tilde{\phi}_v^n \le x)}{\P(\max_{v \in T_n} \tilde{\phi}_v^n \le X)} dx \\
		&+& \frac{1}{\sqrt{2\pi}\sigma_{d,n}}\int_{m_n-a\log n }^{\infty} x e^{-x^2/2\sigma^2_{d,n}} \dfrac{\P(\max_{v \in T_n} \tilde{\phi}_v^n \le x)}{\P(\max_{v \in T_n} \tilde{\phi}_v^n \le X)} dx 
	\end{eqnarray*}
	We denote the first term by $J_1$ and the next one by $J_2$. We first want to show that the contribution of $J_2$ in the conditional expectation is negligible. We use a trivial upper bound on the tail probability in the numerator. Then we compute the integral which is the tail expectation of a normal.  
	\begin{eqnarray*}
		J_2 &\le& \frac{1}{\sqrt{2\pi}\sigma_{d,n}}\int_{m_n-a\log n }^{\infty} x e^{-x^2/2\sigma^2_{d,n}} \dfrac{1}{\P(\max_{v \in T_n} \tilde{\phi}_v^n \le X)} dx \\
		&=& \frac{1}{\sqrt{2\pi}\sigma_{d,n}\P(\max_{v \in T_n} \tilde{\phi}_v^n \le X)}\int_{m_n-a\log n }^{\infty} x e^{-x^2/2\sigma^2_{d,n}} dx \\
		&=& \frac{\sigma_{d,n}e^{-(m_n-a\log n)^2/2\sigma^2_{d,n}}}{\sqrt{2\pi}\P(\max_{v \in T_n} \tilde{\phi}_v^n \le X)}  
	\end{eqnarray*}
	So we end up showing that contribution from the right tail is negligible. We now move on to the rest part and obtain an upper bound for it. We use a general upper bound on $x$ from the range of the integral, which we can do since the integral exists and is finite by the fact that absolute expectation of a normal exists. 
	\begin{eqnarray*}
		J_1 &\le& \frac{m_n-a\log n }{\sqrt{2\pi}\sigma_{d,n}}\int^{m_n-a\log n }_{-\infty}  e^{-x^2/2\sigma^2_{d,n}} \dfrac{\P(\max_{v \in T_n} \tilde{\phi}_v^n \le x)}{\P(\max_{v \in T_n} \tilde{\phi}_v^n \le X)} dx \\
		&\le& \frac{m_n-a\log n }{\sqrt{2\pi}\sigma_{d,n}}\int^{\infty }_{-\infty}  e^{-x^2/2\sigma^2_{d,n}} \dfrac{\P(\max_{v \in T_n} \tilde{\phi}_v^n \le x)}{\P(\max_{v \in T_n} \tilde{\phi}_v^n \le X)} dx \\ 
		&=& m_n-a\log n
	\end{eqnarray*}
	From (\ref{eq-positivity}) it is clear that on choosing $a$ such that $a\log n \le \lambda'$ the upper bound on the conditional expectation is $m_n - a \log n$. Hence we can choose $a=\frac{\sqrt{2}}{\sqrt{\log d}}$.
	
	{\bf Lower Bound: } We apply a similar technique as in case of the upper bound, the only difference being that we look at the left tail instead, motivated by the left tail of the maximum of the Gaussian process. 
	\begin{eqnarray*}
		\E\left(X \mid \max_{ v \in T_n} \tilde{\phi}_v^n \le X \right) &=& \frac{1}{\sqrt{2\pi}\sigma_{d,n}}\int_{-\infty}^{\infty} x e^{-x^2/2\sigma^2_{d,n}} \dfrac{\P(\max_{v \in T_n} \tilde{\phi}_v^n \le x)}{\P(\max_{v \in T_n} \tilde{\phi}_v^n \le X)} dx \\ 
		&=& \frac{1}{\sqrt{2\pi}\sigma_{d,n}}\int_{-\infty}^{m_n-\frac{3}{c}\log_d n} x e^{-x^2/2\sigma^2_{d,n}} \dfrac{\P(\max_{v \in T_n} \tilde{\phi}_v^n \le x)}{\P(\max_{v \in T_n} \tilde{\phi}_v^n \le X)} dx \\
		&+& \frac{1}{\sqrt{2\pi}\sigma_{d,n}}\int_{m_n-\frac{3}{c}\log_d n}^{\infty} x e^{-x^2/2\sigma^2_{d,n}} \dfrac{\P(\max_{v \in T_n} \tilde{\phi}_v^n \le x)}{\P(\max_{v \in T_n} \tilde{\phi}_v^n \le X)} dx
	\end{eqnarray*}
	We denote the first term by $I_1$ and the second by $I_2$.
	
	When $x \in (-\infty, m_n-\frac{3}{c}\log_d n]$ then $ \P(\max_{v \in T_n} \tilde{\phi}_v^n \le x) \le C' \exp(-C''n^3)$ following (\ref{eq-lefttailssbrw}). Also we have a lower bound on the probability of positivity, which gives the following bounds on $I_1$ and $I_2$. 
	\begin{eqnarray*}
			\vert I_1 \vert \leqslant \underset{\sim}{C} e^{\frac{1}{2\sigma^2_{d,n}}(m_n - \lambda')^2+ d^{c\lambda'}(\log \lambda'  - \log \bar{p}/K) - C''n^3} \int_{-\infty}^{\infty} \mid x \mid e^{-x^2/2\sigma^2_{d,n}} dx 
		\end{eqnarray*}
		where $\underset{\sim}{C}>0$ is a constant not depending on $n$. This shows that this term is negligible. 
	Further, 
	\begin{eqnarray*}
		I_2 &\ge& \left(m_n-\frac{3}{c}\log_d n \right)\frac{1}{\sqrt{2\pi}\sigma_{d,n}}\int_{m_n-\frac{3}{c}\log_d n}^{\infty}  e^{-x^2/2\sigma^2_{d,n}} \dfrac{\P(\max_{v \in T_n} \tilde{\phi}_v^n \le x)}{\P(\max_{v \in T_n} \tilde{\phi}_v^n \le X)} dx  \\
		&=& \left(m_n-\frac{3}{c}\log_d n \right)\frac{1}{\sqrt{2\pi}\sigma_{d,n}}\int_{-\infty}^{\infty}  e^{-x^2/2\sigma^2_{d,n}} \dfrac{\P(\max_{v \in T_n} \tilde{\phi}_v^n \le x)}{\P(\max_{v \in T_n} \tilde{\phi}_v^n \le X)} dx - o(1)  \\
		&=& m_n-\frac{3\sqrt{2}}{\sqrt{\log d}}\log n - o(1). 
	\end{eqnarray*} 
	In the last step we have used the value of $c=\frac{1}{\sqrt{2 \log d}}$, as fixed before in Lemma \ref{lem-lefttail}. 
\end{proof}

\appendix

\section{Acknowledgement}



\noindent The author would like to thank his graduate advisor Prof. Jian Ding for suggesting the problem and for many helpful discussions on the same.



%
%
%
%

\end{document}